\documentclass[a4paper]{article}
\usepackage[english]{babel}
\usepackage[utf8]{inputenc}
\usepackage{amsmath}
\usepackage{amsthm}
\usepackage{amssymb}
\usepackage{breqn}
\usepackage{graphicx}
\usepackage{fancyhdr}
\usepackage{subfigure}
\usepackage{subfig}
\usepackage{setspace}
\usepackage{caption}
\usepackage[margin=1in]{geometry}
\usepackage[T1]{fontenc}
\usepackage{authblk}

\newtheorem{theorem}{Theorem}
\newtheorem{lemma}[theorem]{Lemma}
\newtheorem{cor}[theorem]{Corollary}

\theoremstyle{plain}

\theoremstyle{definition}

\theoremstyle{remark}

\title{Applications of Information Theory:\\
statistics and statistical mechanics}
\author[1]{Khizar Qureshi \thanks{Prof. Peter Shor provided generous amounts of feedback\\
Department of Mathematics\\
Massachusetts Institute of Techno\text{log}y\\
18.424: Se\text{min}ar in Information Theory}}

\date{\today}
\fancypagestyle{phdthesis}{%
\fancyhf{}
\fancyhead[R]{\thepage}
}

\begin{document}
\onehalfspacing
\maketitle
\begin{abstract}

The method of optimizing entropy is used to (i) conduct Asymptotic Hypothesis Testing and (ii) deter\text{min}e the energy distribution for which Entropy is maximized. This paper focuses on two related applications of information theory: Statistics and Statistical Mechanics.

\end{abstract}
\section*{Introduction}
Entropy is one measure of uncertainty within a system, and is often used to describe the disorder of sequences of quantized random variables. However, entropy can also be extended to methods within optimization, in which the disorder of a system of interest may be maximized or \text{min}imized. Such methods are prevalent within statistics, the physical sciences, and econometrics. 
\paragraph{}
The concerns of a statistician observing a sequence of outcomes include the validity of an explanatory hypothesis, its degree of significance, and any assumptions underlying the statistical tests. While linear hypothesis testing is often sufficient, larger sequences exhibit large deviations in behavior that should receive separate treatment. Traditional linear hypothesis testing trivially assigns a constant multiple of an explanatory parameter $\beta$ to an observation when for\text{min}g a hypothesis ($H_0 = k\beta = 0$). Optimal entropy, in which disorder is locally \text{min}imized or maximized, can be used to construct \textit{asymptotic}, non-linear hypothesis tests. Unlike linear hypothesis testing, error probability can be \text{min}imized.

Optimizing entropy extends to thermodynamic systems. The \textit{Third Law of Thermodynamics} states that the entropy of a closed system, i.e. \ one in which no mass or energy is added or removed, must be bounded from below by zero. Achieving a non-entropic system is nearly impossible, except within a perfect crystal lattice. A more probable state is one for which the entropy of a system is maximized, and observations follow a Boltzmann distribution.

In our study of optimal entropy, we will use classical Statistics and Statistical Mechanics as a lens. We will demonstrate the concept of \text{min}imal entropy through two statistical tests: the univariate optimality test  defined by \textit{Stein's Lemma}, and a multivariate optimality test, as defined by the \textit{Chernoff Bounds}. We will see that there exists a distribution, the Boltzmann distribution, that approaches maximum entropy as temperature goes to infinity. Finally, we will apply the concept of asymptotic hypothesis testing to Statistical Mechanics. In particular, we will test and observe the evolution of error probability with a growing sample size. We will also compare Q-function error probability with that of the Chernoff bound. The remainder of the paper is organized as follows. To better understand the atypicality of sequences, we will study the method of types. We will then learn how such sequences behave through the large deviation theory. The focus of the paper will then shift to hypothesis testing, in which we will develop tools for recognizing the asymptotic optimality of entropy. Illustrative examples of this will include Stein's Lemma and Chernoff bounds. For the interest of the physical sciences, we will rigorously derive the Boltzmann distribution, for which entropy is nearly maximized. Finally, we will converge the aforementioned topics through simulations of asymptotic statistical testing.

\section*{Hypothesis Testing}
Statisticians are often concerned with not just observed data, but the several possible underlying explanations. A few examples include: testing for the effectiveness of a drug, deter\text{min}ing whether or not a coin is biased, and the effect of gender on wage growth. We begin with a simple case in which we decide between two hypothesis, each of which is represented by an independent and identical distribution, or i.i.d. Let $X_1, X_2, \ldots, X_n$ be i.i.d. $\sim Q(x)$. For an observed outcome, we have two possible explanations:
\begin{itemize}
\item $H_1: Q = P_1$
\item $H_2: Q= P_2$
\end{itemize} 
We now define a general decision function, whose value reflects the acceptance and rejection of the above hypothesis. Namely, for general decision $g(x_1, x_2, \ldots, x_n)$, $g(x)=i$ indicates that $H_i$ is accepted. In the binary case, the set $A$ over which $g(x)=i$ is complemented by the set $A^c$, over which $g(x) \neq i$. 
\paragraph{}
Quite often, statisticians are concerned with accepting incorrect hypotheses and rejecting correct ones. Such occurrences, recognized as Type I/II errors, often occur when sequences exhibit atypicality and large deviating behavior (See appendix). Error probabilities are reflected through the decision function using weights $\alpha, \beta$:
\begin{equation}
\begin{split}
\alpha = P(g(x) = 2 | H_1 \text{ true}) = P^n_1(A^c)\\
\beta = P(g(x) = 1 | H_2 \text{ true}) = P^n_1(A)
\end{split}
\end{equation}
Notice that the general decision function takes on values contradicting those implied by the conditional hypothesis. The first implies that $H_2$ was accepted even though $H_1$ was true, and the second implies that $H_1$ was accepted even though $H_2$ was true. Type I (reject true) and Type II (accept false) errors similarly prove detrimental to experiments, and so we wish to minimize probabilities $\alpha$ and $\beta$.  Minimizing $\alpha$ increases $\beta$, and minimizing $\beta$ increases $\alpha$. We will now explore methodology to minimize the overall probability of error by optimizing entropy as a weighted sum of $\alpha$ and $\beta$. 

\subsection*{Stein's Lemma}
We first fix either $\alpha$ or $\beta$, and manipulate the other to \text{min}imize the probability of error.  
\begin{theorem}[Stein's Lemma]
Let $X_1, X_2, \ldots, X_n$ be i.i.d. $\sim Q$. Further, let $D(P_1 \| P_2)$ represent the Kullback-Leibler distance, or relative entropy between the probability densities. Consider the hypothesis test between two alternatives $Q=P_1$ and $ Q=P_2$ where $D(P_1 \| P_2)<\infty$. Let $A_n \subseteq \mathcal H^n$ be an acceptance region for hypothesis 1. Let the probabilities of error be
\begin{equation}
\begin{split}
\alpha_n = P^n_1(A^c_n)\\
\beta_n = P^n_2(A_n)
\end{split}
\end{equation} 
and for $0<\epsilon<\frac{1}{2}$, define
\begin{equation}
\beta^{\epsilon}_n = \text{min}_{A_n \subseteq X^n} \beta_n
\end{equation}
Then,
\begin{equation}
\text{lim}_{\epsilon \rightarrow 0} \  \text{lim}_{n \rightarrow \infty} \frac{1}{n} \text{log} \beta^{\epsilon}_n = -D \left( P_1 \| P_2 \right)
\end{equation}
\begin{proof}
See Appendix
\end{proof}
\end{theorem}
Thus, no sequence of sets $B_n$ has an exponent better than $D \left( P_1 \| P_2 \right)$. But the sequence $A_n$ achieves the exponent $D \left( P_1 \| P_2 \right)$. Thus $A_n$ is asymptotically optimal, and the best error exponent is $D \left( P_1 \| P_2 \right)$.

\subsection*{Chernoff Bound}
Thus far, $\alpha$ and $\beta$ have been treated separately. The approach underlying Stein's Lemma was to set one error probability to be infinitesimally small, and measure the effect on the resulting probability. We saw that setting $\alpha \leq \epsilon$ achieved $\beta_n = 2^{-nD}$. However, the distribution of error amongst $\alpha$ and $\beta$ may be highly asymmetrical, in which case univariate optimization may not suffice. We now explore methodo\text{log}y for a bivariate optimization.  
\paragraph{}
An alternative approach is to \text{min}imize the weighted sum of $\alpha$ and $\beta$. The resulting error exponent is known as the \textit{Chernoff Information}. Consider a distribution of i.i.d. random variables: $X_1, X_2, \ldots, X_n$ representative of the decision function. We assign $P_1$ to Q with probability $\pi_1$ and $P_2$ to Q with probability $\pi_2$. Upholding the definition of $\alpha$ and $\beta$, the overall probability of error is
\begin{equation}
P^n_{\epsilon} = \pi_1 \alpha_n + \pi_2 \beta_n
\end{equation}
\begin{theorem}[Chernoff]
The best achievable exponent in the Bayesian probability of error is $D^{*}$, where 
\begin{equation}
D^* = \\text{lim}_{n \rightarrow \infty} \text{min}_{A_n \subseteq \mathcal{H}^n} - \frac{1}{n} \text{log} P^n_{\epsilon} = D \left( P_{\lambda^*} \| P_1 \right) = D\left( P_{\lambda^*} \| P_2 \right)
\end{equation}
where
\begin{equation}
P_{\lambda} = \frac{P^{\lambda}_1(x) P^{1-\lambda}_2(x)}{\sum_{a \in X} P^{\lambda}_1(a) P^{1-\lambda}_2(a)}
\end{equation}
and $\lambda^{*}$ the value of $\lambda$ such that
\begin{equation}
D \left(P_{\lambda^*} \| P_1 \right) = D \left(P_{\lambda^*} \| P_2 \right).
\end{equation}
\begin{proof}
See Appendix
\end{proof}
\end{theorem}

\section*{Physical Chemistry}
Claude Shannon first proposed that the uncertainty due to possible errors in a message could be encapsulated by
\begin{equation}
U(W) = \text{log} W
\end{equation}
where W is the number of possible ways (state space) of encoding random information. Intuitively, the uncertainty increases with increasing W, and is zero if W=1. The concept of entropy provides a deep-rooted link between information theory and statistical mechanics. The state with the least information available, or greatest entropy occurs when the set of all states are equiprobable. This is also the state with maximum uncertainty. An information theoretic perspective dictates that explicit knowledge of various probabilities associated with the system constitutes greater information. Similarly, the thermodynamics of a system of isolated particles indicate that entropy is directly correlated with expected energy level. Below is a molecular orbital diagram that illustrates the possible energy states, all of which depend on the position an electron occupies.
\begin{figure}[h]
\centering
\includegraphics[scale=0.7]{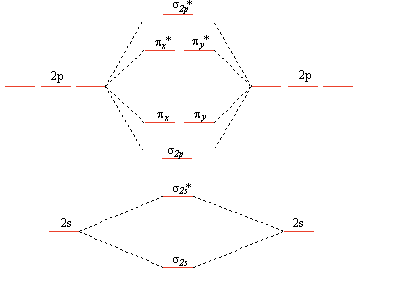}
\caption{The figure above is a molecular orbital diagram. When two atoms interact, and possibly share their valence (outermost) electrons, the electrons must occupy a particular orbital. Once occupying an orbital, the atoms are able to form a bond. Here, we are not concerned with the type of the bond, but rather, paired occupancy. The empty orbitals are \textit{indistinguishable}, and will be occupied with equal probability. }
\end{figure}
\paragraph{}
Entropy may also be observed in macroscopic states. The \textit{second law of thermodynamics} states that in equilibrium, changes in entropy are proportional to changes in system heat per unit temperature. We have $dW = \frac{dQ_{sys}}{T}$. We can better understand this through an illustration. Consider a system of gas particles that may be expanded or compressed. We can study the system under various entropic regimes. The diagram below illustrates how available work decreases (increases) for gaseous expansion (compression) under bivariate states of pressure and volume.
\begin{figure}[h]
\centering
\includegraphics[scale=0.5]{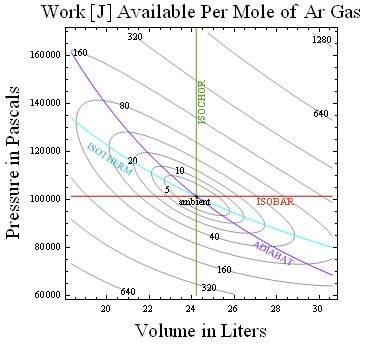}
\caption{The figure above is a pressure-volume diagram for a system of Argon gas particles. Expanding or compressing a gas requires energy, the extent of which depends on the state of the system. Notice that for an adiabatic system $\left(dQ_{sys} = 0 \right)$, compressing the gas requires the least relative energy. That is, when when the change in entropy in \text{min}imized, a system can be most naturally expanded/compressed. To reach \text{min}imum work available, we move down the gradient of steepest descent until entropy is globally \text{min}imized.}
\end{figure}
\paragraph{}
Consider a perfectly structured crystal lattice structure, in which the positions of each contributing molecule is fixed. If we observe such a system, depart, and return after $n$ periods, the position of each molecule within the crystal will have remained the same \textit{almost surely}. If the particles did not displace, then they also carried zero kinetic energy, which is representative of a zero temperature system. This near certainty of a thermodynamic system is an example of an optimization in which entropy is \text{min}imized. If instead, the system consisted of a fair coin toss, with no extra information, entropy would be maximized.   

\subsection*{The Boltzmann Distribution}
The distribution that maximizes the state space for a fixed energy level is the Boltzmann distribution. We will now derive such a distribution, and show that it uniquely maximizes entropy on each energy level. Consider a crystal containing $N$particles, each of which has available energy levels, $\epsilon_n$. The state space, W, is the number of ways the total energy $E = \sum_n N_n\epsilon_n$ can be distributed amongst the the particles in each energy level, across all energy levels, $N=\sum_n N_n$. The expected number of particles in each energy level, $N_n$, is the product of the probability that a particle is at an energy level, $P_n$, and the total number of particles, N. The only consideration for such a distribution is the number of particles in each energy levels, not necessarily the amount of energy allocated to each particle. Energy is conserved amongst states and across energy levels. While there exist several ways of assigning the number of particles in each energy level $\epsilon_n$, we wish to find the state with the distribution achievable in the most number of ways for fixed energy levels. Our first constraint is that the total state space W is the sum of individual states occupied, $W_i$ for all possible distributions
\begin{equation}
W = \sum W_i
\end{equation}
Amongst all possible distributions of particles, there exists one that can be achieved in more ways than any other. A distribution that approaches maximum entropy for fixed energy levels is the Boltzmann distribution.

\begin{equation}
\text{log} W \cong \text{log} W_B
\end{equation}
To find the most probable distribution that maximizes W, we first note that each particle in the crystal can be distinguished from the others because it occupies a defined position in space. Therefore, such a setting allows us to number the particles $1,2, \ldots, N$. We assume a large N, to maintain consistency with typical non-deficient states. S  A particular microstate of the crystal will place particle 1 in energy level $\epsilon_i$, particle 2 in energy level $\epsilon_j$, and so on. Initially, we seek the number of $W_i$ microstates in a distribution for which there are $N_1$ particles in $\epsilon_1$, $N_2$ particles in $\epsilon_2,$, and so on. We choose, at random, particles from the crystal, $N_i$, and assign them to energy levels, $\epsilon_i$. The number of ways this can be done is equal to the number of different orders in which the particles can be chosen from the crystal. The first particle can be chosen from a group of N. With $N-1$ particles remaining, the second can be chosen in $N-1$ ways. We see that the number of ways for selecting the first two particles is $N(N-1)$. Following this procedure, we then we see that the number of ways for selecting the $N$ particles is $N(N-1)(N-2)(N-3) \dots (3)(2)(1)$, or $N!$. 
\paragraph{}
We have over counted the ways of achieving a given distribution, and have assumed that all states are distinguishable. Consider the placement of the first two particles into energy level $\epsilon_1$. It makes no difference whether the first particle is placed into $\epsilon_1$ prior to or following particle 2. That is, the states are $\textit{indistinguishable}$. This relaxes the strictness on order, and so permutation are ignored. Thus, the state space, $W_i$, for a given distribution, is $N!$ divided by the product of all $N!$
\begin{equation}
W_i = \frac{N!}{\prod_{n} N_n!}
\end{equation} 
To find the distribution that maximizes $W_i$, we note Stirling's approximation for \text{log} N!
\begin{equation}
\text{log} \ N! \approx N \ \text{log} \ N - N
\end{equation}

Finding the maximum of $W_i$ is equivalent to finding the maximum of $\text{log} \ W_i$, so we combine equations (13) and (14), and re-arrange as follows
\begin{equation}
\text{log} \ W_i \approx \text{log} \ N! - \sum_n \text{log} N_n! = N \ \text{log} \ N - \sum_n N_n \ \text{log} N_n = \left(\sum_n N_n \right) \ \text{log} \sum_n N_n - \sum_n N_n \ \text{log} N_n
\end{equation}
However, the set of particles is conserved. Moreover, the net energy within the system is conserved. This provides the following two constraints
\begin{equation}
\begin{split}
\sum_j N_j = N\\
\sum_j \epsilon_j N_j = E
\end{split}
\end{equation}
We now use Lagrange's method of undeter\text{min}ed multipliers. When $W_i$ is maximized, its differential \text{log} must be zero
\begin{equation}
d \ \text{log} \ W_i = \sum_j \frac{\partial \text{log} W_i}{\partial N_j}_j dN_j =0
\end{equation}
We multiply the constraints on particle count and energy by constants $\alpha, \beta$, and then take the differential to obtain
\begin{equation}
\begin{split}
\alpha \sum_j dN_j = 0\\
\beta \sum_j \epsilon_j dN_j = 0
\end{split}
\end{equation}
Subtracting these two constraints from the \text{log}-entropy, we obtain
\begin{equation}
\sum_j \left(\frac{\partial \text{log} W_i}{\partial N_j} - \alpha -\beta \epsilon_j \right) dN_j = 0
\end{equation}
Through use of \text{log}-properties and algebraic manipulation (See Appendix), the expression above is reduced to
\begin{equation}
\text{log} N_j = \text{log} N - \alpha -\beta \epsilon_j
\end{equation}
which, after exponentiating both sides is
\begin{equation}
N_j = Ne^{-\alpha}e^{-\beta \epsilon_j}
\end{equation}
The significance of this result is that it shows the occupancy of en energy state $\epsilon_j$ is proportional to $e^{-\beta \epsilon_j}$. 
\subsubsection*{Key Result}
To account for energy states, thermodynamicists often make use of temperature, an intrinsic quantity. Temperature is equivalent to the average kinetic energy of a system of particles. Because this varies across systems, we normalize. For a temperature T, and the Boltzmann constant, $k_B$, 
\begin{equation}
\beta = \frac{1}{k_BT}.
\end{equation}
Inducting on the one particle case, in which, $e^{\alpha} = \sum_j e^{\frac{-\epsilon_j}{k_BT}}$, and combining the preceding three expressions, providing the desired result
\begin{equation}
\begin{split}
P_n = \frac{N_n}{N} = \frac{e^z}{\sum_N e^z}\\
\text{where } z=\frac{-\epsilon_j}{k_BT}
\end{split}
\end{equation}
We have now found the Boltzmann distribution, for which entropy is maximized. The Boltzmann probability above expresses the fraction of particles placed in each quantum state $n$ to maximize entropy $W_i$ of the distribution over each energy level, $\epsilon_j$.
\section*{Simulation: Asymptotic Hypothesis Testing}
Thus far, we have studied various methods of statistical testing, highlighting the importance of asymptotic tests such as the Chernoff Information Bound. We have also (briefly) explored Statistical Mechanics, in which we show that entropy is maximized for a Boltzmann distribution. We now demonstrate the importance of our learnings through a representative example. 
\paragraph{}
Robust methods of signal interpretation allow for communication, and involve the separation of \textit{signal} and \textit{noise}. A simple signal will follow a Gaussian distribution, for which entropy is maximized. A hypothesis consists of assigning observations as either signals, or as noise. Such hypotheses carry error probabilities, and should be studied with both linear testing, as well as asymptotic testing. 

\subsection*{Example: Binary Detection}
The classical binary detection problem involves the reception of finite-length signals realized as a random process $r[n], n=1,2, \ldots,$ [3]. The signal can be attributed to either Gaussian white noise $n[i]$ or a deterministic signal $s[i]$. Basic studies involve the interpretation of the signal-to-noise ratio, a measure of quality.
Consider a binary detection problem:
\begin{equation} 
r=s_i + n, ,\ i \in \{1,2 \}
\end{equation}
\begin{itemize}
\item Detections are composed of signals and noise
\item n: N-dimensional noise vector 
\item i.i.d. Gaussian random variables and $\sim \left(0,1 \right)$
\item $s_1 = \left(m,m,\ldots,m \right)$
\item $s_2 = \left(0,0,\ldots 0 \right)$
\end{itemize}
Evaluate the error probability for both $N=1$ and $N=4$ when $m \in \{1,2,3,4,5,6 \}$.\\

Traditionally, error probability is evaluated through the Q-function, which represents the probability that a normal random variable will obtain a value larger than $x$ standard deviations above the mean. It can also be thought of as the "tail" probability of the standard normal distribution, and is useful for linear hypothesis testing.

\begin{equation}
Q(x) = \frac{1}{\sqrt{2 \pi}} \int _x^{\infty} exp \left(\frac{-t^2}{2} \right) dt=\frac{1}{2} erf \left( \frac{x}{\sqrt{2}} \right)
\end{equation}
Given a signal-to-noise Ratio, the Q-function can be used to deter\text{min}e the error probability
\begin{equation}
P^{e} = Q \left(\frac{\rho}{\sigma} \right) \text{where } \ \frac{\rho^2}{\sigma^2} = \frac{Nm^2}{4}
\end{equation}

We also know that any error probability is bounded from above by the Chernoff Information bound. For the Q-function,
\begin{equation}
Q(x) \leq exp \left(-\frac{x^2}{2} \right) \text{erf}
\end{equation}
And so,
\begin{equation}
P^{e} = Q \left( \frac{\sqrt{N}m}{2} \right) \leq exp \left(-\frac{m^2N}{8} \right)
\end{equation}
The figure below illustrates the growth of error in both forms of testing.

\begin{figure}[htb!]
\centering
\includegraphics[scale=0.6]{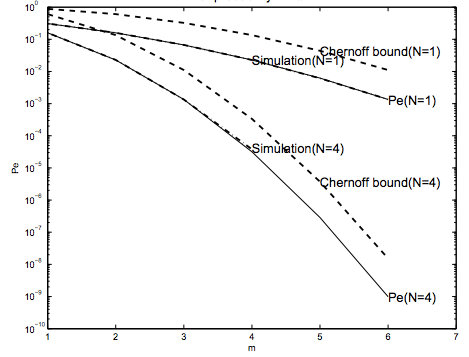}
\caption{The figure above shows the evolution of error probability with an increasing signal length (m). Error probability decreases as: (i) the number of bits in the signal increases, and (ii) the number of elements in the noise vector increases.}
\end{figure}
\section*{Concluding Remarks}
Optimizing entropy demonstrates the applicability of information theory beyond computing. Asymptotic testing captures error probability in atypical sequences, and a Boltzmann distribution of particles approaches maximum entropy, as temperature goes to infinity. \newpage
\section*{Appendix}
Motivation for asymptotic testing arises from atypicality and large deviations in sequences. We \textit{briefly} review this, and encourage the ambitious reader to study further.
\subsection*{The Method of Types}

The Asymptotic Equipartition Property formalizes that although there exist several possible outcomes of a stochastic process, there exists a set from which sequences are \textit{typical} , or most frequently observed. The centric approach underlying the AEP involves defining an almost sure convergence in probability between the expectation of a sequence to its entropy. Similarly, the Method of Types defines strong bounds on the number of sequences of a particular distribution, as well as the probability of each such sequence being observed. 

\section*{Large Deviation Theory}

Recall that \textit{type} of a sequence $x^n_i \in A^n$ is representative of its empirical distribution $\hat{P} = \hat{P}_{x^n_i}$ where:
\begin{equation}
\hat{P}(a) = \frac{|\{i: x_i = a\}|}{n}, a \in A.
\end{equation}
A distribution P on A is called an \textit{n-type} if it is the type of some $x^n_1 \in A^n$. The set of all $x^n_1 \in A^n$ of type P is called the \textit{type class of the n-type} P and is denoted by $\mathcal{T}^n_p$. 
\begin{lemma}
The number of possible n-types is
\begin{equation}
\binom{n + |A| - 1}{|A| - 1}
\end{equation}
\end{lemma}

\begin{proof}
\begin{equation}
\mathcal{T}(P) = \{x \in \mathcal{X}^n : P_x = P\}
\end{equation}
The combinatoric cardinality of $\mathcal{T}(P)$ provides the result
\end{proof}
\begin{lemma}
For any n-type P, 
\begin{equation}
\binom{n + |A| -1}{|A| - 1}^{-1} 2^{nH \left(P\right)} \leq |\mathcal{T}^n_p| \leq 2^{nH \left(P\right)}
\end{equation}
\end{lemma}
\begin{proof}
First, we prove the upper bound using $P(\mathcal{T}(P) \leq 1$.
\begin{equation}
1 \geq P^n (\mathcal{T}(P)) = \sum_{x\ in T(P)} P^n (x) = \sum_{x\ in T(P)} 2^{-nH(P)} = \|T(P)\|2^{-nH(P)}
\end{equation}
Consequently, $\|T(P)\| \leq 2^{nH(P)}$. For the lower bound, using the fact that $T(P)$ has the highest probability amongst all type classes in P, we can bound the ratio of probabilities
\begin{equation}
\frac{P^n(T(P))}{P^n(T(\hat{P}))}\\
= \frac{\|T(P)\| \prod_{a\in \mathcal{X} P(a)^{nP(a)}}}{\|T(\hat{P})\| \prod_{a\in \mathcal{X} P(a)^{n\hat{P}(a)}}}
=\prod_{a \in \mathcal{X}} \frac{(n \hat{P}(a))!}{(n{P}(a))!} P(a)^{n(P(a) -\hat{P}(a))}\\
\end{equation}
Using the identity $\frac{m!}{n!} \geq n^{m-n}$, we see 
\begin{equation}
\frac{P^n (\mathcal{T}(P))}{P^n (\mathcal{T}(\hat{P}))} \geq \prod_{a \in \mathcal{X}} n^{n(P(a) - \hat{P}(a))} = n^{n(1-1)}=1
\end{equation}
So $P^n(T(P)) \geq P^n (T(\hat{P}))$. The lower bound can now be found as
\begin{equation}
\begin{aligned}
1 = \sum_{Q \in \mathcal{P}_n} P^n (T(Q)) \leq \sum_{Q \in \mathcal{P}_n} \\
= max_Q P^n (T(Q))\\
 = \sum_{Q \in \mathcal{P}_n} P^n (T(P))\\
\leq (n+1)^{\| \mathcal{X} \|} P^n (T(P)) = (n+1)^{\| \mathcal{X} \|} \sum_{x \in T(P)} P^n (x)\\
 = (n+1)^{\| \mathcal{X} \|} \sum_{x \in T(P)} 2^{-nH(P)} \\
=(n+1)^{\| \mathcal{X} \|} \|T(P) \| 2^{-nH(P)}
\end{aligned}
\end{equation}

\end{proof}
To connect the theory of types with general probability theory, we must develop a sense of relative entropy. For any distribution P on A, let $P^n$ denote the distribution of n independent drawings from P, that is, $P^n \left(x^n_1 \right) = \prod_{i=1}^n P \left(x_i \right), x^n_1 \in A^n$.
\begin{lemma}
For any distribution P on A and any \textit{n-type} Q
\begin{equation}
\begin{split}
\frac{P^n \left(x^n_1 \right)}{Q^n \left(x^n_1 \right)} = 2^{-nD \left(Q \| P \right)}, if x^n_1 \in \mathcal{T}^n_Q\\
\binom{n+ |A| -1}{|A|-1}^{-1} 2^{-nD \left(Q \| P \right)} \leq P \left(\mathcal{T}^n_p \right) \leq 2^{-nD \left(Q \| P \right)}
\end{split}
\end{equation}
\begin{proof}
For probability $P \in P_n$, distribution $Q$, the probability of type class $T(P)$ under $Q^n$ is $2^{-nD(P \| Q)}$. We see 
\begin{equation}
\begin{aligned}
Q^n(T(P)) = \sum_{x \in T(P)} Q^n (x) \\ 
= \sum_{x \in T(P)} 2^{-n(D(P \|Q) + H(P))} \\
= \| T(P) \| 2^{-n(D(P \| Q ) + H(P))}
\end{aligned}
\end{equation}
Replacing $\| T(P) \| $ with the result from Lemma 3, we see the result.
\end{proof}
\end{lemma}
\begin{cor}
Let $\hat{P}_n$ denote the empirical distribution (type) of a random sample of size n drawn from P. Then
\begin{equation}
P \left(D \left( \hat{P}_n \| P \right) \geq \delta \right) \leq \binom{n+ |A| -1}{|A|-1} 2^{-n \delta}, \forall \delta >0
\end{equation}
\end{cor}
\begin{proof}
Given an $\epsilon > 0$, we can define a typical set $\mathcal{T}^{\epsilon}_Q$ of sequences for the distribution Q as $\mathcal{T}^{\epsilon}_Q = \{x^n: D(P_{x^n}\|Q) \leq \epsilon\}$. Then the probability of an atypical sequence is
\begin{equation}
\begin{split}
1-Q^n(\mathcal{T}^{\epsilon}_Q) = \sum_{P: D(P \| Q) \geq \epsilon} Q^n (T(P)) \\
\leq \sum_{P: D(P \| Q) \geq \epsilon} 2^{-nD (P \| Q)} \\
\leq \sum_{P: D(P \| Q) \geq \epsilon} 2^{-n \delta} \\
\leq \binom{n+ |A| -1}{|A|-1} 2^{-n \delta}, \forall \delta \geq 0
\end{split}
\end{equation}
\end{proof}

\begin{theorem}{Sanov's Theorem}
Let $\Pi$ be a set of distributions on A whose closure is equal to the closure of its interior. Then for the empirical distribution of a sample from a strictly positive distribution P on A, 
\begin{equation}
-\frac{1}{n} \text{\text{log}} P \left(\hat{P} \in \Pi \right) \rightarrow D \left(\Pi \| P \right)
\end{equation}
\end{theorem}
\begin{proof}{Sanov's Theorem}
Let $\mathcal{P}_n$ be the set of possible n-types and let $\Pi_n = \Pi \cap \mathcal{P}_n$. The previous lemma implies that
\begin{equation}
\begin{split}
Prob \left (\hat{P}_n \in \Pi_n \right) = P^n \left( \cup_{Q\in \Pi_n} \mathcal{T}^n_Q \right) \ \text{is upper bounded by}\\
\binom{n+|A|-1}{|A|-1}2^{-nD \left(\Pi_n \| P \right)} \ \text{and lower bounded by}\\
\binom{n+|A|-1}{|A|-1}^{-1}2^{-nD \left(\Pi_n \| P \right)} 
\end{split}
\end{equation}
Since $D \left(Q \| P \right)$ is continuous in Q, the hypothesis on $\Pi$ implies that $D \left( \Pi_n \| P \right)$ is arbitrarily close to $D \left(\Pi \| P \right)$ if n is large. 
\end{proof}

\subsection*{Proofs}
\begin{proof}[Stein's Lemma]
To prove the theorem, we construct a sequence of acceptance regions $A_n \subseteq X^n$ such that $A_n < \epsilon$ and $\beta_n = 2^{-nD(P_1 \| P_2)}$. We then show that no other sequence of tests has an asymptotically better exponent.\\
First, we define
\begin{equation}
A_n = \left\{x \in X^n: 2^{+n(D(P_1 \| P_2)^{-\delta})} \leq \frac{P_1x}{P_2x} \leq 2^{+n(D(P_1 \| P_2)^{+\delta}}\right\}
\end{equation}
Then, we have the following properties:
\begin{enumerate}
\item $P^n_1(A_n) \rightarrow 1$. This follows from:
\begin{equation}
P^n_1(A_n) = P^n_1 \left(\frac{1}{n} \sum_{i=1}^n \text{log} \frac{P_1(X_i)}{P_2(X_i)} \in \left(D(P_1 \| P_2 \right) - \delta_1D \left(P_1 \| P_2 \right) + \delta \right)
\end{equation} by the strong LLN, since $D\left(P_1 \| P_2 \right) = E_{P_1} \left(\text{log} \frac{P_1(X)}{P_2(X)} \right)$. Hence, for sufficiently large n, $A_n < \epsilon$.
\item $P^n_2(A_n) \leq ^{+n(D(P_1 \| P_2)^{-\delta})}$. Using the definition of $A_n$, we have
\begin{equation}
\begin{split}
P^n_2(A_n) = \sum_{A_n} P_2(x) \leq \sum_{A_n} P_1(x) 2^{+n(D(P_1 \| P_2)^{-\delta})}\\
=2^{+n(D(P_1 \| P_2)^{-\delta})} \sum_{A_n} P_1(x)\\
=2^{+n(D(P_1 \| P_2)^{-\delta})}(1-\alpha_n).
\end{split}
\end{equation}
Similarly, $P^n_2(A_n) \geq 2^{+n(D(P_1 \| P_2)^{+\delta})}(1-\alpha_n)$. And so, $\frac{1}{n} \text{log} \beta_n \leq -D \left(P_1 \| P_2 \right) + \delta + \frac{\text{log}(1-\alpha_n)}{n}$, and $\frac{1}{n} \text{log} \beta_n \geq -D \left(P_1 \| P_2 \right) - \delta + \frac{\text{log}(1-\alpha_n)}{n}$, hence $ \ \text{lim}_{n \rightarrow \infty} \frac{1}{n} \text{\text{log}} \beta_n = -D \left(P_1 \| P_2 \right).$
\end{enumerate}
\end{proof}
\begin{proof}[Chernoff Bound]
The optimum hypothesis test is a likelihood ratio test, which follows the form: 
\begin{equation}
D\left( P_{X^n} \| P_2 \right) -D \left(P_{X^n} \| P_1 \right) >T
\end{equation}
The test divides the probability simplex into regions corresponding to hypothesis 1 and hypothesis 2, respectively. This is illustrated below:
\begin{figure}[ht]
\centering
\includegraphics[scale=0.5]{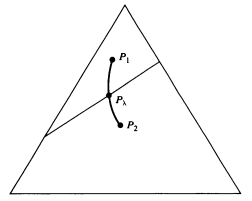}
\caption{The figure above shows the probability simplex and the Chernoff bound. Notice that for error probabilities $P_1$ and $P_2$, there exists an optimal error probability $P_{\lambda}$. This is deter\text{min}ed as a weighted arg\text{min}.}
\end{figure}
Let A be the set of types associated with hypothesis 1. From the preceding discussions, it follows that the closest point in the set $A^c$ to $P_1$ is on the boundary of A, and is of the form given by $[8]$. Then, it is clear that $P_{\lambda}$ is the distribution in A that is closest to $P_2$. It is also the distribution in $A^c$ that is closest to $P_1$. By Sanov's theorem, we can calculate the associated probabilities of error:
\begin{equation}
\begin{split}
\alpha_n = P^n_1(A^c) = 2^{-nD\left(P_{\lambda^*} \| P_1 \right)}\\
\beta_n = P^n_2(A)  = 2^{-nD\left(P_{\lambda^*} \| P_2 \right)}
\end{split}
\end{equation}
In the Bayesian case, the overall probability of error is the weighted sum of the two probabilities of error,
\begin{equation}
P_e = \pi_1 2^{-nD\left(P_{\lambda^*} \| P_1 \right)} + \pi_2 2^{-nD\left(P_{\lambda^*} \| P_2 \right)} = 2^{-n \ \text{min} \{{-D\left(P_{\lambda^*} \| P_1 \right),-D\left(P_{\lambda^*} \| P_2 \right)  \}}}
\end{equation}
since the exponential rate is deter\text{min}ed by the worst exponent. Since $D\left(P_{\lambda} \| P_1 \right)$ increases with $\lambda$ and $D \left(P_{\lambda} \|P_2 \right)$ decreases with $\lambda$, the maximum value of the \text{min}imum of $\{ D \left( P_{\lambda} \| P_1 \right), D\left(P_{\lambda} \| P_2 \right) \}$ is attained when they are equal. We choose $\lambda$ so that 
\begin{equation}
D\left(P_{\lambda} \| P_1 \right) = D \left(P_{\lambda} \| P_2 \right) = C(P_1, P_2)
\end{equation}
Thus $C\left(P_1, P_2\right)$ is the highest achievable exponent for the probability of error, and is called the Chernoff information.

The closest point in the set $A^c$ to $P_1$ is on the boundary of A, and is of the form given by []. Then from the previous discussion, it is clear that $P_{\lambda}$ is the distribution in A that is closest to $P_2$; it is also the distribution in $A^c$ that is closest to $P_1$. By Sanov's theorem, we can calculate the associated probabilities of error:
\begin{equation}
\begin{split}
\alpha_n = P^n_1 \left(A^c \right) = 2^{-nD \left (P_{\lambda^*} \| P_1 \right)}\\
\beta_n =  P^n_2 \left(A^c \right) = 2^{-nD \left (P_{\lambda^*} \| P_2 \right)}
\end{split}
\end{equation}
In the Bayesian case, the overall probability of error is the weighted sum of the individual two probabilities of error,
\begin{equation}
P_{e} = \pi_1 2^{-nD \left( P_{\lambda} \| P_1 \right)} + \pi_2 2^{-nD \left( P_{\lambda} \| P_2 \right )} = 2^{-n \ \text{min} \left( D(P_{\lambda} \| P_1 \right), D \left(P_{\lambda} \| P_2 \right) }
\end{equation}
since the exponential rate is deter\text{min}ed by the worst exponent. Since $D \left( P_{\lambda} \| P_1 \right) = D \left ( P_{\lambda} \| P_1 \right)$ increases with $\lambda$ and $D \left( P_{\lambda} \| P_2 \right)$ decreases with $\lambda$, the maximum value of the \text{min}imum of $D \left( P_{\lambda} \| P_1 \right), D \left( P_{\lambda} \| P_2 \right)$ is attained when they are equal. 

This is illustrated below:

\begin{figure}[h]
\centering
\includegraphics[scale=0.5]{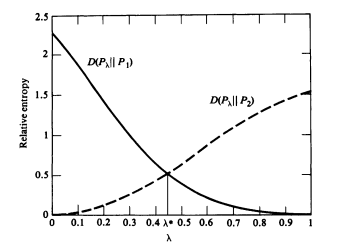}
\caption{The figure above shows the relative entropy for each error probability as a function of $\lambda$.}
\end{figure}

We choose $\lambda$ so that 
\begin{equation}
D\left(P_{\lambda} \| P_1 \right) = D \left(P_{\lambda} \| P_2 \right) = C(P_1, P_2)
\end{equation}
Thus $C\left(P_1, P_2\right)$ is the highest achievable exponent for the probability of error, and is called the Chernoff information.

\end{proof}

\end{document}